\documentclass[11pt,oneside]{amsart}
\usepackage[maxbibnames=99, giveninits=true, doi=false, url=false, isbn=false]{biblatex}
\usepackage[english]{babel}
\usepackage{csquotes}
\usepackage{geometry}
\usepackage{amsmath}
\usepackage{amssymb}
\usepackage{amsthm}
\usepackage{xfrac}

\usepackage{enumitem}
\usepackage{booktabs}
\usepackage{dcolumn}
\usepackage{fancyhdr}
\usepackage{graphicx}
\usepackage{hyperref}
\hypersetup{colorlinks,allcolors=black}
\usepackage{latexsym}
\usepackage{todonotes}
\usepackage{siunitx}
\usepackage{tipa}
\usepackage{stackengine}
\usepackage{mathrsfs}
\usepackage{breqn}
\usepackage{pxfonts}
\usepackage{verbatim}

\makeatletter
\newtheoremstyle{numberfirst} % nome dello stile
  {3pt}% spazio sopra
  {3pt}% spazio sotto
  {\itshape}% corpo del teorema
  {}% indentazione
  {\bfseries}% font per titolo
  {.}% punteggiatura dopo titolo
  {0.5em}% spazio dopo titolo
  {\thmnumber{#2} \thmname{#1}\thmnote{ \textnormal{(#3)}}}% formato titolo
\makeatother

\newtheorem{theorem}{Theorem}[section]
\newtheorem{lemma}[theorem]{Lemma}
\newtheorem{proposition}[theorem]{Proposition}
\newtheorem{corollary}[theorem]{Corollary}

\theoremstyle{definition}

\newtheorem{remark}[theorem]{Remark}

\makeatletter
\renewcommand{\section}{\@startsection{section}{1}%
  \z@{2.5ex plus 1ex minus .2ex}{1.3ex plus .2ex}%
  {\normalfont\Large\scshape\centering}}

\renewcommand{\thesubsection}{\thesection.\arabic{theorem}}

\makeatletter
\let\oldsubsection\subsection
\renewcommand{\subsection}[1]{%
  \refstepcounter{theorem}% incrementa il contatore dei teoremi
  \oldsubsection*{\thesubsection\kern0.35em #1}% <-- spacing matches theorem label
  \addcontentsline{toc}{subsection}{\thesubsection\quad #1}% per l’indice
}
\makeatother

\numberwithin{equation}{section}

\newcommand{\R}{\mathbb{R}}
\newcommand{\F}{\mathcal{F}}
\newcommand{\eps}{\varepsilon}
\newcommand{\dB}{\partial B}

\DeclareMathOperator{\bari}{bar}

\newcommand{\blue}[1]{{\color{blue}#1}}

\title{Stability of the Ball in Isoperimetric Inequalities Between Two Fractional Perimeters}

\author{Giovanni Alberti\textsuperscript{1}, Giacomo Cozzi\textsuperscript{2}, Annalisa Massaccesi\textsuperscript{2}, Jeremy Mirmina\textsuperscript{1}}
\address{G. Alberti, J. Mirmina: \textsuperscript{1}Dipartimento di Matematica, Universit\`a di Pisa, Largo Bruno Pontecorvo, 5, 56127 Pisa PI, Italy}
\address{G. Cozzi, A. Massaccesi: \textsuperscript{2}Dipartimento di Matematica ``Tullio Levi-Civita'', Universit\`a degli Studi di Padova, Via Trieste, 63, 35131 Padova PD, Italy}

\begin{document}
\begin{abstract}
    We consider the isoperimetric inequality involving the $s$-perimeter and the $t$-perimeter with $0<s<t<1$, and show that the ball is a local minimizer of the (scale-invariant) isoperimetric ratio $$\F(E)\coloneqq \frac{P_t(E)^{\frac{1}{n-t}}}{ P_s(E)^{\frac{1}{n-s}}}$$ among sets $E$ that are nearly spherical. To this end, we rewrite $\F$ as a functional of $u$, where $u$ is a scalar function on the unit sphere in $\R^n$ that parametrizes $\partial E$, and prove a quantitative stability result for $\F$ around $u=0$ with respect to a suitable Sobolev norm (Theorem \ref{main theorem quantitative}).
    This parallels known results where the $s$-perimeter is replaced by the volume.
\end{abstract}
\maketitle
\setcounter{tocdepth}{1} % Only sections in ToC
\tableofcontents

\section{Introduction}
For $\alpha \in(0,1)$, the fractional $\alpha$-perimeter of a Borel set $E\subset\R^n$ is defined as the square of the $H^{\frac{\alpha}{2}}$-seminorm of the characteristic function of $E$, that is,
$$
P_\alpha(E)\coloneqq
\frac{1}{2} \int_{\mathbb{R}^n} \int_{\mathbb{R}^n} \frac{\left|\chi_E(x)-\chi_E(y)\right|^2}{|x-y|^{n+\alpha}} d x d y=\int_E \int_{E^c} \frac{d x d y}{|x-y|^{n+\alpha}}.
$$

The notion of fractional perimeters was introduced in \cite{CaffRoqSav2010,Visintin91}, and represents a \textit{nonlocal} analogue of the classical perimeter functional, as it takes into account interactions between  points in the set and points in the complement arbitrarily far apart.
The functional $P_\alpha(E)$ can be thought of as a $(n-\alpha)$-dimensional perimeter  in the sense that $P_\alpha(\lambda E) = \lambda^{n-\alpha}P_\alpha(E)$ for any $\lambda>0$. 
% parag 3: Ps vs |*|
One may see the class of fractional perimeters $\{P_\alpha\}_{\alpha\in(0,1)}$ as a collection of functionals ranging between the usual notions of volume ($\alpha=0$) and perimeter ($\alpha=1$) of a set. Indeed, let us recall that (see \cite{ambrosio2010gamma}, \cite{BouBreMir2002}, \cite{davila2002open}) for any $E\subset \R^n$ of class $C^{1,\gamma}$ for some $\gamma>0$,
\begin{equation*}
    \lim_{\alpha\to 1}\,\frac{1-\alpha}{P(B)}\,P_\alpha(E)=P(E),
\end{equation*}
where $P(E)$ denotes the classical perimeter of the set $E$. On the other hand (see \cite{maz2002bourgain}, \cite{dipierro2014strongly}),
\begin{equation*}
    \lim_{\alpha \to 0} \,\frac{\alpha}{P(B)}P_\alpha(E)=|E|,
\end{equation*}
where $|E|$ stands for the Lebesgue measure of $E$. The nonlocal analogue of the classical isoperimetric inequality relates the $\alpha$-perimeter with the volume and it is by now well understood: the problem consists in finding the optimal constant $C_{n,\alpha}$ for which it holds 
\begin{equation}
\label{alpha isoperimetric inequality}
   |E|^{\frac{1}{n}}\leq C_{n,\alpha}P_\alpha(E)^{\frac{1}{n-\alpha}} \hspace{5mm} \text{for every } E\subset \R^{n}.
\end{equation}
Using the Riesz rearrangement inequality, it can be shown that $C_{n,\alpha}$ is the constant achieved by the ball and that equality holds if and only if $E$ is the ball (see \cite[Theorem 4.1]{frakSeiringer08}).
A quantitative version of this inequality, formulated in terms of Fraenkel asymmetry (see Corollary \ref{main corollary} for the definition), was later established in \cite{FMM2011}, and subsequently refined with the sharp decay estimate in \cite{figalli2015isoperimetry}. 

% Ps vs Pt
In this paper we investigate a natural extension of \eqref{alpha isoperimetric inequality}. Namely, given $0<s<t<1$ we want to find the optimal constant $C_{n,s,t}$ for which it holds
\begin{equation}\label{eq fractional isoperimetric inequality}
     P_s(E)^{\frac{1}{n-s}}\leq C_{n,s,t}P_t(E)^{\frac{1}{n-t}} \hspace{5mm} \text{for every } E \subset \R^n.
\end{equation}
We remark that, by fractional Sobolev embeddings (see for instance Theorem 1.1 in \cite{diplipetal}), one can prove that \eqref{eq fractional isoperimetric inequality} holds for some (possibly non - optimal) constant $C_{n,s,t}$. 
The problem of finding the optimal constant in \eqref{eq fractional isoperimetric inequality} is equivalent to the problem of finding the minimum of the scale -, rotation - and translation - invariant functional,
\begin{equation}
\label{eq: def of F(E)}
    \F_{s,t}(E)\coloneqq\frac{P_t(E)^{\frac{1}{n-t}}}{P_s(E)^{\frac{1}{n-s}}},
\end{equation} namely to show that
\begin{equation*}
    \F_{s,t}(E)\geq \frac{1}{C_{n,s,t}} \hspace{5mm} \text{for every } E \subset \R^n.
\end{equation*}
When the parameters $s,t$ have been fixed, we will simply denote $\F_{s,t}$ by $\F$.
In \cite{di2015nonlocal}, it is shown that the functional $\F$ admits a minimum and that every minimizer is bounded with boundary of class $C^{1,\beta}$ for some $\beta\in(0,1)$.
We expect that balls are the unique global minimizers of $\F$ for all $0< s<t< 1$. This would mean that the sharp constant in  \eqref{eq fractional isoperimetric inequality} is $C_{n,s,t}=\F(B)^{-1}$.

\noindent The main result of the present work establishes a local version of this conjecture, namely the local stability of the ball for the functional $\F$. More precisely, we prove that balls are the unique minimizers of $\F$ among competitors that are small $C^1$ graph perturbations of a ball.
% parag definition of graph perturbation
Let us clarify what we mean by $C^1$ perturbation of a ball: let $B(y,r)$ denote the Euclidean ball centered at $y\in \R^n$, with radius $r>0$, and let $B$ denote the unit ball at the origin. We say that $E$ is a \emph{graph over the ball} $B(y,r)$ if there exists a $C^1$ function $u\colon \partial B\to \R$ such that 
\begin{equation}
\label{graph over the ball}
 E\coloneqq \{y+\lambda(1+ u(x))x \colon \lambda \in[0,r], x \in \partial B\}.
\end{equation}
In the literature these sets are also known as \textit{nearly spherical sets}.

We prove the following
\begin{theorem}\label{main theorem}
    Let $0<s<t<1$ and $n\geq2$. Then
    there exists $\eps_1>0$, depending only on $s,t,$ and $n$, such that, if $E$ is a graph over any ball via a function $u\in C^1(\partial B)$, with $\|u\|_{C^1}\leq\eps_1$, then 
    \begin{equation*}
        \F(E)\geq\F(B),
    \end{equation*}
    and equality holds if and only if $E$ is a ball.
\end{theorem}
Actually, Theorem \ref{main theorem} is a consequence of a more precise statement that goes in the direction of a quantitative isoperimetric inequality:

\begin{theorem}\label{main theorem quantitative}
    Let $0<s<t<1$ and $n\geq2$. Then, there exist constants $\eps_0, c_0 >0$, depending only on $s,t,$ and $n$, such that the following holds. Let  $E$ be a set, and let $B(y,r)$ be the ball having the same barycenter and volume as $E$. If $E$ is a graph over $B(y,r)$ via a function $u$ satisfying $\|u\|_{C^1}\leq\eps_0$, then
    \begin{equation}\label{eq2 main theor}
        \mathcal{F}(E)\geq \mathcal{F}(B)+c_0\|u\|^2_{H^{\frac{1+t}{2}}}.
    \end{equation}
\end{theorem}
Theorem \ref{main theorem} in turn implies the following
\begin{corollary}
\label{main corollary}
    If $E\subset\R^n$ is as in Theorem \ref{main theorem quantitative}, then there exists $c>0$ such that  
    \begin{equation}
    \label{eq main corollary}
        \mathcal{F}(E)\geq \mathcal{F}(B)+cA(E)^2,
    \end{equation}
    where 
    \begin{equation*}
        A(E)\coloneqq \inf_{x\in\R^n}\left\{\frac{|E\triangle B_x|}{|B_x|}\colon |B_x|=|E|\right\}.
    \end{equation*}
    \end{corollary}
    The quantity $A(E)$ is commonly known as the Fraenkel asymmetry of $E$ and it measures the $L^1$-distance of $E$ from a ball.
    Theorem \ref{main theorem quantitative} simply implies Corollary \ref{main corollary} from the fact that the $H^{\frac{1+t}{2}}$-norm of a function controls its $L^1$-norm. Observe moreover that the exponent $2$ in \eqref{eq2 main theor} and in \eqref{eq main corollary} is sharp. 
    An inequality of the form \eqref{eq main corollary} is the sharp isoperimetric inequality that we conjecture to hold true even when considering arbitrary sets $E\subset \R^n$. While replacing $\|u\|_{H^{\frac{1+t}{2}}}$ with the Fraenkel asymmetry weakens the inequality at the perturbative level, this choice is motivated by the geometric significance of $A(E)$ and by the fact that this is the most common and historically relevant notion of asymmetry in the literature.

\oldsubsection*{Remarks and open problems}
\begin{remark}
    Observe that \eqref{P_alpha(B_u)} implies that
    $P_\alpha(E_u)\leq C\left(\|u\|_{L^\infty}
    +[u]^2_{H^{\frac{1+\alpha}{2}}}\right)$ , for some $C>0$.
    Therefore, the conditions  $u\in H^{\frac{1+\alpha}{2}}(\partial B)$ and $\|u\|_\infty\leq\frac{1}{2}$ are sufficient to ensure that $P_\alpha(E_u)<+\infty$.
\end{remark}

\begin{remark}\label{rem constraints}
    In Theorem \ref{main theorem quantitative}, we impose the constraints on volume and barycenter in order to ensure a \textit{quantitative} estimate: the volume constraint excludes constant non-zero functions $u$, which would simply correspond to a rescaled ball (so the quantitative estimate would fail). One could remove this assumption by considering the homogeneous Sobolev seminorm instead of the full norm. However, even with this modification, the barycenter constraint would still be needed to rule out translations of the function $u$, that is, translations of the ball itself.
\end{remark}

\begin{remark}
Consider the rescaled functional
    \begin{equation*}
        \F^{*}_{s,t}(E)\coloneqq\frac{\left(\frac{1-t}{|\partial B|}P_t(E)\right)^{\frac{1}{n-t}}}{\left(\frac{s}{|\partial B|}P_s(E)\right)^{\frac{1}{n-s}}}.
    \end{equation*} 
    As we said above, when $s\to 0$ and $t\to1$, this quantity converges to
    $\text{P}(E)^{\frac{1}{n-1}}/|E|^{\frac{1}{n}}$.
    As a consequence the locally optimal constant $C_{n,s,t}$ in \eqref{eq fractional isoperimetric inequality} that we obtain when $E=B$, converges as $s\to0$ and $t\to1$ to the optimal constant in the classical isoperimetric inequality.
    A similar consideration holds true also for the single limit $s\to 0$.
\end{remark}

\begin{remark}\label{remark cicalese leonardi}
    When proving quantitative isoperimetric inequalities, the typical ingredients are the following  (see for example \cite{fuglede1989stability}, \cite{fusco2008sharp}):
    \begin{enumerate}[label=$(\roman*)$]
        \item a \textit{strict qualitative inequality}, showing that the ball is the unique minimizer;
        \item a reduction to the small asymmetry regime, establishing that it is sufficient to prove \eqref{eq main corollary} for sets whose $L^1$-distance from a ball is small;
        \item a \textit{Fuglede-type stability result}, showing  the quantitative isoperimetric inequality in the class of ``nearly spherical'' sets, \textit{i.e.}, sets that are $C^1$ graphs over the unit ball.
    \end{enumerate}
    From this, and also applying a selection principle in the spirit of \cite{CicaleseLeonardi-SelectionPrinciple}, one might get the global quantitative inequality in the form of \eqref{eq main corollary}.
    Theorem \ref{main theorem quantitative} represents step $(iii)$ in the above program.
\end{remark}

\begin{remark}
    Unlike the classical case, fractional isoperimetric inequalities remain  meaningful even in dimension $1$, both for $P_t$ versus $P_s$ and for $P_t$ versus volume. 
    It would be interesting to understand if points $(i)$ and $(ii)$ of Remark \ref{remark cicalese leonardi} could be more accessible. 
\end{remark}
\begin{remark}
Observe that symmetrization techniques, which are classical tools in the proof of isoperimetric inequalities and in the characterization of balls as optimizers, are still effective when dealing with the inequality between the $s$-perimeter and the volume. Indeed, rearrangement preserves the volume while decreasing the $s$-perimeter. 
By contrast, for the inequality \eqref{eq fractional isoperimetric inequality}, these techniques are no longer helpful, since both the $s$ and $t$ perimeters decrease under rearrangement and it is not possible to determine whether the $s$-perimeter decreases ``more'' than the $t$-perimeter.
\end{remark}
\textbf{Sketch of the proof.}
Our approach relies on a second-order analysis of the functional $\F$. 
Two key estimates play a central role: the coercivity (Proposition \ref{prop coercivity}) and a weak form of continuity for the second variation (Proposition \ref{prop dam-lam estimate}).

The coercivity estimate is established in the $H^{\frac{1+t}{2}}$-norm, which is the natural and optimal norm for the problem. In contrast, the continuity estimate requires additional control in the stronger $C^1$-norm. More precisely, we obtain a bound of the form
\begin{equation*}
|\delta^2\F(u)[u,u]-\delta^2\F(0)[u,u]|\leq C_{n,s,t} \,\|u\|_{C^1}\|u\|^2_{H^{\frac{1+t}{2}}}.    
\end{equation*} 
This discrepancy between the norms is not specific to our setting, but a common feature in shape-variation problems. For example, it already appears for the classical perimeter, which is naturally differentiable in $W^{1,\infty}$, whereas coercivity may hold only in the $H^1$-norm. 
A detailed and clear discussion of this phenomenon can be found in \cite{dambrine2019stability}.
\\

\textbf{Organization of the paper.}
Section \ref{preliminaries} provides some preliminary tools and presents a useful formula describing the Sobolev norm of a function defined on the sphere, in terms of its spherical harmonic decomposition.
In Section \ref{section first and second variations} we explicitly compute the first and second variations of the functional $\F$ in a neighborhood of the origin. 
The two key estimates of  \textit{coercivity} and \textit{continuity along rays} for the second variation of $\F$ are established in Sections \ref{section coercivity} and \ref{section continuity}, respectively. Finally, we provide the proofs of Theorem \ref{main theorem} and Theorem \ref{main theorem quantitative} in Section \ref{section geometric results}.

\section{Preliminaries}
\label{preliminaries}

Let $B$ denote the closed unit ball in $\R^n$, and $\partial B$ its boundary. For $\alpha\in(0,1)$, we define the fractional Sobolev space $H^\alpha(\partial B)$ as
\begin{equation*}
    H^{\alpha}(\partial B)=\left\{u\in L^2(\partial B)\colon [u]^2_{H^\alpha(\partial B)} <+\infty\right\},
\end{equation*}
where $[u]^2_{H^\alpha(\partial B)}$ is the Gagliardo seminorm of $u$ and it is defined as
\begin{equation*}
    [u]^2_{H^\alpha(\partial B)}:=\int_{\partial B \times \partial B} \frac{(u(x)-u(y))^2}{|x-y|^{n-1+2\alpha}}d\sigma_xd\sigma_y.
\end{equation*}
Here $d\sigma$ denotes the surface measure on the sphere.
$H^\alpha(\partial B)$ is a Hilbert space if endowed with the norm
\begin{equation*}
    \|u\|_{H^{\alpha(\partial B)}}\coloneqq \left(\|u\|^2_{L^2(\partial B)}+[u]^2_{H^\alpha(\partial B)}\right)^{1/2}.
\end{equation*}
It is well known that, for $0 < s < t < 1$, the embedding $
H^t(\partial B) \hookrightarrow H^s(\partial B)$ is continuous,
meaning that $H^t(\partial B)\subset H^s(\partial B)$ and it holds $\|u\|_{H^s(\partial B)} \le C \, \|u\|_{H^t(\partial B)}$ for a constant $C$ depending only on $n,s,t$. For general references on fractional Sobolev spaces, see for instance \cite{leoni-fractionalsobolevbook,HitchhikerfractionalSobolev}.

We now turn our attention to spherical harmonics and Sobolev functions on the sphere. \label{Spherical harmonics and Sobolev functions on the sphere}
For $k \in \mathbb{N}$, let $d(k)$ denote the dimension of the space of homogeneous harmonic polynomials of degree $k$ over $\R^n$, namely the set of polynomials $P$ of the form
\begin{equation*}
    P(x)=\sum_{j=0}^{J}\, c_{j}\, x^{\alpha_j},\hspace{5mm} \Delta P(x)=0, 
\end{equation*}
with $J\in \mathbb{N}$ and $\alpha_j$ a multi-index such that $|\alpha_j|=k$ for any $j\in \{0,\ldots,J\}$. In particular, $d(0)=1$ and $d(1)=n$.
The spherical harmonics of degree $k$ are the collection $\mathcal{S}_k= \left\{Y^i_k\right\}_{i=1}^{d(k)}$, obtained as the restriction to $\partial B$ of an orthonormal basis (with respect to the inner product on $\partial B$) of the homogeneous polynomials of degree $k$ on $\R^n$ (notice that, therefore, $Y_0^1=|\partial B|^{-1/2}$). One can prove that $\{\mathcal{S}_k\}_{k\in \mathbb{N}}$ is an orthonormal basis for $L^2(\partial B)$.
For a given $u\colon \partial B \to \R$ of class $L^2$, let us denote by $a^i_k(u)$ the Fourier coefficient of $u$ corresponding to $Y^i_k$, namely $a^i_k(u)=\int_{\partial B } u(x)Y^i_k(x)\, dx$. Then, we can write
\begin{equation*}
    u(x)=\sum_{k\in \mathbb{N}}\sum_{i=1}^{d(k)} a_k^i(u)Y^i_k(x),
\end{equation*}
and 
\begin{equation*}
\begin{split}
    \|u\|^2_{L^2(\partial B)}&=\sum_{k\in \mathbb{N}}\sum_{i=1}^{d(k)} a_k^i(u)^2,\\
    \int_{\partial B} u(x)\,d\sigma&=\sum_{k\in \mathbb{N}}\sum_{i=1}^{d(k)}a_k^i(u)=|\partial B|^{\frac{1}{2}}a^1_0(u).
\end{split}
\end{equation*}
Moreover, following \cite{figalli2015isoperimetry}, we have that
\begin{equation}
    \label{SobolevNorm}
    [u]^2_{H^{\frac{1+\alpha}{2}}(\partial B)}=\sum_{k=0}^{\infty}\sum_{i=1}^{d(k)}\lambda_{k,\alpha} \,a^i_k(u)^2,
\end{equation}
where
\begin{equation}
    \label{autovalori seminorma sphere}
   \lambda_{k,\alpha} = \frac{2^{1-\alpha}\pi^{\frac{n-1}{2}}}{1+\alpha}\frac{\Gamma\left(\frac{1-\alpha}{2}\right)}{\Gamma\left(\frac{n+\alpha}{2}\right)}\left(\frac{\Gamma\left(k+\frac{n+\alpha}{2}\right)}{\Gamma\left(k+\frac{n-2-\alpha}{2}\right)}-\frac{\Gamma\left(\frac{n+\alpha}{2}\right)}{\Gamma\left(\frac{n-2-\alpha}{2}\right)}\right),
\end{equation}
and $\Gamma$ is the Euler Gamma function.

From now on, and for the rest of the paper, we will consider sets $E$ that are graph perturbations of the unit ball $B$ of some function $u\in C^1(\partial B)$. Recalling definition \eqref{graph over the ball}, this means that sets we will consider are of the form 
\begin{equation*}
 E\coloneqq \{\lambda(1+ u(x))x \colon \lambda \in[0,1], x \in \partial B\}.
\end{equation*}
When $u$ is given, we will call such a set $E_u$.

In the following lemma we show that, for any function $u$ with the property that the perturbed set $E_u$ has its barycenter at the origin and the same volume as the unit ball, the coefficients $a_0^0(u)$ and $a_1^i(u)$ for $0\leq i \leq n$ can be controlled by the square of the $L^2$-norm of $u$.

\begin{lemma}\label{lemma control a_0 and a_1 with L_2 norm}
    Let $u\in L^\infty(\partial B)$ with $\|u\|_{\infty}\leq 1/2$. There exist constants $C_0(n),C_1(n)>0$ such that
    \begin{enumerate}[label=(\roman*)]
        \item If $|E_u|=|B|$, then $\left|a^1_0(u)\right| \leq C_0(n)\|u\|^2_{L^2}$;    
        \item If $E_u$ has barycenter in the origin, then $\left|a^i_1(u)\right| \leq C_1(n)\|u\|^2_{L^2}$, for any $1\leq i \leq n$.
    \end{enumerate}
\end{lemma}

\begin{proof}
    We prove {\it (i)}. Assume that $|E_u|=|B|$ and recall that 
\begin{equation}\label{eq: int(1+u)^n-1=0}
    \int_{\partial B}(1+u)^n \, d\sigma=n|E_u|=n|B|=\mathcal{H}^{n-1}(\partial B).
\end{equation}
We consider the function
\begin{equation*}
f(t)=\int_{\partial B}[(1+tu(x))^n-1]\,d\sigma.
\end{equation*}
Expanding $f$ via Taylor near $t=0$, we obtain
\begin{equation*}
    f(1)=n\int_{\partial B} u(x)d\sigma+\frac{n(n-1)}{2}\int_{\partial B}(1+\lambda u(x))^{n-2}u(x)^2d\sigma,
\end{equation*}
for a certain $\lambda \in (0,1)$.
By \eqref{eq: int(1+u)^n-1=0} we have $f(1)=0$, therefore
\begin{equation*}
    \int_{\partial B} u(x)d\sigma=-\frac{n-1}{2}\int_{\partial B}(1+\lambda u(x))^{n-2}u(x)^2d\sigma.
\end{equation*}
From this we deduce
\begin{equation*}
        \left|a^1_0(u)\right|=\frac{1}{|\partial B|^{\frac{1}{2}}}\left|\int_{\partial B}u(x)\,d\sigma\right|\leq \frac{(n-1)}{2|\partial B|^{\frac{1}{2}}}(1+\|u\|_{\infty})^{n-2}\|u\|^2_{L^2}\leq C_0(n)\|u\|^2_{L^2}.
\end{equation*}
Analogously, we prove {\it (ii)}. Assume that the barycenter of $E_u$ is in the origin, that is
$\int_{E_u}x \,dx=0$.
For any $i\in d(1)$, one has
\begin{equation}\label{eq: xint(1+u)^n+1=0}
    0=\text{bar}^i(E_u)=\frac{1}{n+1}\int_{\partial B} x_i (1+u(x))^{n+1}\, d\sigma.
\end{equation}
As before, consider the function
\begin{equation*}
    f(t)=\int_{\partial B} x_i (1+tu(x))^{n+1}\, d\sigma,
\end{equation*}
and expand it via Taylor near $t=0$. We obtain
\begin{equation*}
    0=f(1)=(n+1)\int_{\partial B}x_i\,u(x)\, d\sigma+\frac{n(n+1)}{2}\int_{\partial B}x_i(1+\lambda u(x))^{n-1}u(x)^2d\sigma.
\end{equation*}
Hence,
\begin{equation*}
\left|\int_{\partial B}x_i\,u(x)\, d\sigma\right|=\left|\frac{n}{2}\int_{\partial B}x_i(1+\lambda u(x))^{n-1}u(x)^2d\sigma\right|\leq C(n)\|u\|^2_{L^2}.
\end{equation*}
Finally, by definition, for any $0\leq i \leq n$, the spherical harmonic of degree one $Y_1^{i}$ is such that $Y_1^{i}(x)=c(n)x_i$. Therefore,
\begin{equation*}
    \left|a_1^i(u)\right|=\left|\int_{\partial B}Y^i_1(x)\,u(x)\, d\sigma\right|\leq C_1(n)\|u\|^2_{L^2}.\qedhere
\end{equation*}
\end{proof}

\section{Asymptotics of the fractional perimeter near the ball}\label{section first and second variations}

In this section we explicitly compute the first and second variations of the functional of the fractional perimeter $P_\alpha$, which we use to derive the corresponding variations of the functional $\F$ near the unit ball. All the formulas of this section hold under the minimal assumptions of $u\in H^{\frac{1+t}{2}}$ and $\|u\|_{\infty}\leq \frac{1}{2}.$ The first assumption guarantees that $P_\alpha(u)<\infty$, and the second that $E_u$ is a well-defined set.
With a slight abuse of notation we define the functional $P_\alpha(u)\coloneqq P_\alpha(E_u)$ for $\alpha\in(0,1)$ and $\F(u)\coloneqq \F(E_u)$.

First, 
using polar coordinates and rearranging terms (see \cite{figalli2015isoperimetry}),
we can represent the $\alpha$-Perimeter of a nearly spherical set as follows 
\begin{equation}
\begin{split}
    \label{P_alpha(B_u)}
    P_\alpha(u)=&\frac{P_\alpha(B)}{P(B)}\int_{\partial B} (1+u)^{n-\alpha}d\sigma\\
    &+\frac{1}{2}\int_{\partial B \times \partial B}\left(\int_{u(y)}^{u(x)}\int_{u(y)}^{u(x)}F_{|x-y|}(1+r,1+\rho)\,dr \,d\rho\right) \,d\sigma_x\,d\sigma_y,
\end{split}
\end{equation}
where
\begin{equation}
    \label{F_|x-y|}
    F_{|x-y|}(r,\rho)=\frac{(r\rho)^{n-1}}{|rx-\rho y|^{n+\alpha}}=\frac{(r\rho)^{n-1}}{((r-\rho)^2+r\rho|x-y|^{2})^{\frac{n+\alpha}{2}}}.
\end{equation}

\subsection{First variation of $P_\alpha$}

 The first variation of $P_\alpha(u)$ has already been computed in \cite[Lemma 2.1]{degenn2023asymptotic}. We include the computation here for completeness. Recall that, for any $\varphi\in C^1(\partial B)$
    \begin{equation*}
    \delta P_\alpha(u)[\varphi]=\frac{d}{d\lambda}\bigg|_{\lambda=0} P_\alpha(u+\lambda\varphi).
    \end{equation*}
     By differentiating the first term in \eqref{P_alpha(B_u)} under the integral sign, we obtain
\begin{equation*}
    \frac{d}{d\lambda}\bigg|_{\lambda=0}\frac{P_\alpha(B)}{P(B)}\int_{\partial B} (1+u(x)+\lambda\varphi(x))^{n-\alpha}\,d\sigma=\frac{P_\alpha(B)}{P(B)}(n-\alpha)\int_{\partial B} \varphi(x)(1+u(x))^{n-\alpha-1}\,d\sigma.
\end{equation*}
For the second term in \eqref{P_alpha(B_u)}, 
\begin{equation*}
    \begin{split}
   &\frac{1}{2}\frac{\partial}{\partial\lambda}\bigg|_{\lambda=0}\int_{\partial B \times \partial B}\bigg(\int_{u(y)+\lambda\varphi(y)}^{u(x)+\lambda\varphi(x)}
\int_{u(y)+\lambda\varphi(y)}^{u(x)+\lambda\varphi(x)}F_{|x-y|}(1+r,1+\rho)d\rho\,dr\bigg)d\sigma_xd\sigma_y\\
=&\frac{1}{2}\int_{\partial B \times \partial B}\left(\int_{u(y)}^{u(x)}
\left( \varphi(x)F_{|x-y|}(1+r,1+u(x))-
\varphi(y)F_{|x-y|}(1+r,1+u(y)) \right)dr\right)\,d\sigma_x\,d\sigma_y\\
&+ \frac{1}{2}\int_{\partial B \times \partial B}\left(\int_{u(y)}^{u(x)}\frac{\partial}{\partial \lambda}\bigg|_{\lambda=0}\int_{u(y)+\lambda\varphi(y)}^{u(x)+\lambda\varphi(x)}F_{|x-y|}(1+r,1+\rho)d\rho\right)d\sigma_x\,d\sigma_y
\\
=&\frac{1}{2}\int_{\partial B \times \partial B}\left(\int_{u(y)}^{u(x)}
\left(\varphi(x)F_{|x-y|}(1+r,1+u(x))-
\varphi(y)F_{|x-y|}(1+r,1+u(y)) \right)dr\right)d\sigma_x\,d\sigma_y\\
&+\frac{1}{2}\int_{\partial B \times \partial B}\left(\int_{u(y)}^{u(x)}
\left(\varphi(x)F_{|x-y|}(1+u(x),1+\rho)-
\varphi(y)F_{|x-y|}(1+u(y),1+\rho)\right)d\rho\right)d\sigma_x\,d\sigma_y\\
=&2 \int_{\partial B \times \partial B}\left(\int_{u(y)}^{u(x)}
\varphi(x)F_{|x-y|}(1+r,1+u(x))dr\right)d\sigma_x\,d\sigma_y.
\end{split}
\end{equation*}
The last equality follows from the symmetries of the domain $\partial B\times\partial B$ and of the function $F$, namely 
$$F_{|x-y|}(r,\rho)=F_{|x-y|}(\rho,r)=F_{|y-x|}(r,\rho).$$
Hence,
\begin{equation}\label{first variation of P_alpha}
\begin{split}
\delta P_\alpha(u)[\varphi] 
= & (n-\alpha)\frac{P_{\alpha}(B)}{P(B)} 
   \int_{\partial B} \varphi(x)(1+u(x))^{n-\alpha-1}\,d\sigma\\
   &+ 2 \int_{\partial B \times \partial B}\left(\int_{u(y)}^{u(x)} 
      \varphi(x)F_{|x-y|}\left(1+u(x),1+\rho\right)d\rho\right)d\sigma_x\,d\sigma_y.
\end{split}
\end{equation}

\subsection{Second variation of $P_\alpha$}
By differentiating \eqref{first variation of P_alpha}, for any $\varphi,\psi\in C^1(\partial B)$ we obtain 
\begin{equation}\label{second variation of P_alpha}
    \begin{split}
         \delta^2 &P_\alpha(u)[\varphi,\psi]=(n-\alpha)
        (n-\alpha-1)\frac{P_{\alpha}(B)}{P(B)}\int_{\partial B} \varphi(x)\psi(x)(1+u(x))^{n-\alpha-2}\,d\sigma\\       
        &+2\int_{\partial B \times \partial B}\left(\varphi(x)\psi(x)F_{|x-y|}(1+u(x),1+u(x))-\varphi(x)\psi(y)F_{|x-y|}(1+u(y),1+u(x))\right)\,d\sigma_x\,d\sigma_y\\&+2\int_{\partial B \times \partial B}\varphi(x)\psi(x)\left(\int_{u(y)}^{u(x)}\partial_1F_{|x-y|}(1+u(x),1+\rho)d\rho\right)\,d\sigma_x\,d\sigma_y.
        \end{split}
    \end{equation}

\subsection{First variation of $\F$.}
Recalling the expression of $\F$ in \eqref{eq: def of F(E)}, we obtain
\begin{equation}\label{first variation of F}
\begin{split}
    \delta \F(u)[\varphi]&=\frac{1}{n-t}\frac{P_t(u)^{\frac{1}{n-t}-1}}{P_s(u)^{\frac{1}{n-s}}}\delta P_t(u)[\varphi]-\frac{1}{n-s}\frac{P_t(u)^{\frac{1}{n-t}}}{P_s(u)^{\frac{1}{n-s}-1}}\delta P_s(u)[\varphi]\\
    &=\F(u)\left(\frac{\delta P_t(u)[\varphi]}{(n-t)P_t(u)}-\frac{\delta P_s(u)[\varphi]}{(n-s)P_s(u)}\right).
\end{split}
\end{equation}
We observe that the unit ball is a critical point of this functional. Indeed, setting $u=0$, equation \eqref{first variation of P_alpha} reduces to
\begin{equation}\label{first var of P_alpha at zero}
    \delta P_\alpha(0)[\varphi]=(n-\alpha)\frac{P_{\alpha}(B)}{P(B)}\int_{\partial B} \varphi(x).
\end{equation}
Substituting \eqref{first var of P_alpha at zero} into \eqref{first variation of F}, we obtain that \begin{equation}\label{first variation =0}
    \delta  \F(0)[\varphi]=0,
\end{equation}
for every $\varphi\in C^1(\partial B)$.

\subsection{ Second variation of $\F$.} Differentiating equation \eqref{first variation of F}, we get
\begin{equation}\label{second variation of F}
\begin{split}
    &\delta^2 \F(u)[\varphi,\psi]=\F(u)\left(\frac{\delta P_t(u)[\varphi]}{(n-t)P_t(u)}-\frac{\delta P_s(u)[\varphi]}{(n-s)P_s(u)}\right)\left(\frac{\delta P_t(u)[\psi]}{(n-t)P_t(u)}-\frac{\delta P_s(u)[\psi]}{(n-s)P_s(u)}\right)\\
    &+\F(u)\left(\frac{\delta^2 P_t(u)[\varphi,\psi]}{(n-t)P_t(u)}-\frac{\delta^2P_s(u)[\varphi,\psi]}{(n-s)P_s(u)}-\frac{\delta P_t(u)[\varphi]\,\delta P_t(u)[\psi]}{(n-t)P_t(u)^2}+\frac{\delta P_s(u)[\varphi]\,\delta P_s(u)[\psi]}{(n-s)P_s(u)^2}\right).
\end{split}
\end{equation}
Evaluating equation \eqref{second variation of P_alpha} at $u=0$ and $\psi=\varphi$, we get
\begin{equation}\label{second variation of P_alpha at 0}
 \delta^2 P_\alpha(0)[\varphi,\varphi]=
 (n-\alpha)(n-\alpha-1)\frac{P_\alpha(B)}{P(B)}\|\varphi\|_{L^2}^2+[\varphi]^2_{H^{\frac{1+\alpha}{2}}}\leq C(n,\alpha)\|\varphi\|^2_{H^{\frac{1+\alpha}{2}}},
\end{equation}
where $C(n,\alpha)$ is a constant depending only on $n$ and $\alpha$. Then, substituting \eqref{first var of P_alpha at zero} and \eqref{second variation of P_alpha at 0} in \eqref{second variation of F} we get
\begin{equation}\label{second var F at zero}
\begin{split}
    \delta^2 \F(0)[\varphi,\varphi]=\F(0)\left(\frac{[\varphi]^2_{H^{\frac{1+t}{2}}(\partial B)}}{(n-t)P_t(B)}-\frac{[\varphi]^2_{H^{\frac{1+s}{2}}(\partial B)}}{(n-s)P_s(B)}-(t-s)\left(\fint_{\partial B}\varphi^2-\left(\fint_{\partial B} \varphi \right)^2\right)\right).
\end{split}
\end{equation}

\section{Coercivity of the second variation of \texorpdfstring{$\F$}{F}}\label{section coercivity}

The main result is a coercivity result for $\F$ (Proposition \ref{prop coercivity}).
Our proof is based on an explicit computation using the decomposition of $u$ into spherical harmonics.

\begin{proposition}\label{prop coercivity}
Let $u\in H^{\frac{1+t}{2}}(\partial B)$ with $\|u\|_{\infty}\leq\frac{1}{2}$ be such that $|E_u|=|B|$ and $\bari(E_u)=0$. Then there exist two constants $c=c(n,s,t)>0$, and $\eps_0=\eps_0(n)$ such that, if $\|u\|_{L^2}\leq \eps_0$,
\begin{equation}\label{eq:coercivity}
        \frac{1}{2}\delta^2\F(0)[u,u]\geq c\,\|u\|^2_{H^{\frac{1+t}{2}}}.
\end{equation}
\end{proposition}
\begin{remark}
Similarly to what was said in Remark \ref{rem constraints}, the constraints on the volume and barycenter in Proposition \ref{prop coercivity} are necessary to ensure coercivity, as the quadratic form $\delta^2\F(0)[u,u]$ is otherwise only positive semidefinite. These specific constraints are preferred due to their clear geometric interpretation.
\end{remark}

\begin{remark}
    The coercivity of the $\alpha$-Perimeter functional for $\alpha\in(0,1)$ with respect to the $H^{\frac{1+\alpha}{2}}$-norm was established in \cite[Theorem 2.1]{figalli2015isoperimetry}. In our setting $\|\,\cdot\,\|_{H^{\frac{1+t}{2}}}$ remains the natural and optimal norm for the coercivity of the functional $\F$.
\end{remark}

\begin{proof}
\textit{Step 1} In this step we want to rewrite \eqref{second var F at zero}: we will make use of the expression of Sobolev norms of a function in terms of its harmonic coefficients. For the Gagliardo seminorm, recall that (see section \ref{Spherical harmonics and Sobolev functions on the sphere})
\begin{equation}
\label{Gagliardo harmonics}
    [u]^2_{H^{\frac{1+\alpha}{2}}}=\sum_{k\geq 1}\lambda_{k,\alpha} \sum_{i=0}^{d(k)} a_k^i(u)^2,
\end{equation}
where 
\begin{equation}\label{eq lambda^a_k}
    \lambda_{k,\alpha} = \frac{2^{1-\alpha}\pi^{\frac{n-1}{2}}}{1+\alpha}\frac{\Gamma\left(\frac{1-\alpha}{2}\right)}{\Gamma\left(\frac{n+\alpha}{2}\right)}\left(\frac{\Gamma\left(k+\frac{n+\alpha}{2}\right)}{\Gamma\left(k+\frac{n-\alpha-2}{2}\right)}-\frac{\Gamma\left(\frac{n+\alpha}{2}\right)}{\Gamma\left(\frac{n-2-\alpha}{2}\right)}\right).
\end{equation}
We also recall that (see \cite{haddad2022affine})
\begin{equation}\label{eq: P_a(B) in spherical harmonics}
    P_\alpha(B)=\frac{2^{1-\alpha}\pi^{\frac{n-1}{2}}n\omega_n\Gamma\left(\frac{1-\alpha}{2}\right)}{\alpha(n-\alpha)\Gamma\left(\frac{n-\alpha}{2}\right)},
\end{equation}
where $\omega_n=|B|=P(B)/n$. 
Moreover it holds
\begin{equation}
\label{L2 harmonics}
    \fint u^2-\left(\fint u\right)^2=\frac{1}{n\omega_n}\sum_{k\geq 0} \sum_{i=0}^{d(k)} a_k^i(u)^2-\frac{1}{n^2\omega_n^2}a_0^0(u)^2.
\end{equation}
We now substitute \eqref{Gagliardo harmonics} and \eqref{L2 harmonics} in \eqref{second var F at zero} and isolate the coefficients of degree $k=0$ and $k=1$: defining
\begin{equation*}
    A_{\alpha,k}\coloneqq \frac{\lambda_{k,\alpha}}{(n-\alpha)P_\alpha(B)},
\end{equation*}
and noticing that
\begin{equation*}
-\frac{t-s}{n\omega_n}+\frac{\lambda_{1,t}}{(n-t)P_t(B)}-\frac{\lambda_{1,s}}{(n-s)P_s(B)}=\frac{1}{n\omega_n}\left(-(t-s)+A_{t,1}-A_{s,1}\right)=0,
\end{equation*}
we obtain
\begin{equation}
\label{variazione seconda armoniche}
\begin{split}
    \frac{\delta^2\mathcal{F}(0)[u,u]}{\mathcal{F}(0)}&=\sum_{k\geq 2}\left(-\frac{t-s}{n\omega_n} +\frac{\lambda_{k,t}}{(n-t)P_t(B)}-\frac{\lambda_{k,s}}{(n-s)P_s(B)}\right)\sum_{i=0}^{d(k)} a^i_k(u)^2\\
    &+\left(-\frac{t-s}{n\omega_n} +\frac{\lambda_{1,t}}{(n-t)P_t(B)}-\frac{\lambda_{1,s}}{(n-s)P_s(B)}\right)\sum_{i=0}^{d(k)} a^i_1(u)^2+\frac{t-s}{n^2\omega^2_n}a_0^0(u)^2\\
    &= \sum_{k\geq 2}\frac{\sum_{i=0}^{d(k)} a^i_k(u)^2}{n\omega_n} \left(-(t-s)+A_{t,k}-A_{s,k}\right)+\frac{t-s}{n^2\omega^2_n}a_0^0(u)^2\\
    &\geq \sum_{k\geq 2}\frac{\sum_{i=0}^{d(k)} a^i_k(u)^2}{n\omega_n} \left(-(t-s)+A_{t,k}-A_{s,k}\right).
\end{split}
\end{equation} 
To conclude this step, notice that the coefficients $A_{\alpha,k}$ can be written just in terms of the Euler's $\Gamma$ function using \eqref{eq lambda^a_k} and \eqref{eq: P_a(B) in spherical harmonics}:
\begin{equation}
\label{Gamma A_k}
    A_{\alpha,k}=\frac{\alpha}{\alpha+1}\frac{n-\alpha-2}{2}\left(\frac{\Gamma\left(\frac{n-\alpha-2}{2}\right)}{\Gamma\left(\frac{n+\alpha}{2}\right)}\frac{\Gamma\left(k+\frac{n+\alpha}{2}\right)}{\Gamma\left(k+\frac{n-2-\alpha}{2}\right)}-1\right),
\end{equation}
and that
\begin{equation}\label{eq: Ak+1-Ak}
\begin{split}
    A_{\alpha,k+1}-A_{\alpha,k}&=\frac{\alpha}{\alpha+1}\frac{n-\alpha-2}{2n\omega_n}\frac{\Gamma\left(\frac{n-\alpha-2}{2}\right)}{\Gamma\left(k+\frac{n-2-\alpha}{2}\right)}\frac{\Gamma\left(k+\frac{n+\alpha}{2}\right)}{\Gamma\left(\frac{n+\alpha}{2}\right)}\left(\frac{k+\frac{n+\alpha}{2}}{k+\frac{n-2-\alpha}{2}}-1\right)\\
    &=\frac{\alpha}{n\omega_n}\frac{n+\alpha}{2k+n-2-\alpha}\prod_{j=1}^{k-1}\frac{j+\frac{n+\alpha}{2}}{j+\frac{n-\alpha-2}{2}}.
\end{split}
\end{equation}

\textit{Step 2} 
We claim that there exists a constant $c_1=c_1(n,s,t)\in (0,1)$ such that, for any $\lambda > c_1$, it holds
\begin{equation}
\label{Disuguaglianza step 2}
-(t-s)+\lambda A_{t,2}-A_{s,2}>0.
\end{equation}
By \eqref{Gamma A_k}, we have
\begin{equation*}
\begin{split}
    A_{\alpha,2}&=\frac{\alpha}{\alpha+1}\frac{n-\alpha-2}{2}\left(\frac{\Gamma\left(\frac{n-\alpha-2}{2}\right)}{\Gamma\left(\frac{n+\alpha}{2}\right)}\frac{\Gamma\left(2+\frac{n+\alpha}{2}\right)}{\Gamma\left(2+\frac{n-2-\alpha}{2}\right)}-1\right)=\frac{\alpha}{\alpha+1}\frac{n-\alpha-2}{2}\left[\frac{\frac{n+\alpha}{2}\left(\frac{n+\alpha}{2}+1\right)}{\frac{n-2-\alpha}{2}\left(\frac{n-2-\alpha}{2}+1\right)}-1\right]\\
    &=\frac{\alpha}{\alpha+1}\frac{n-\alpha-2}{2}\frac{4n(\alpha+1)}{(n-\alpha)(n-\alpha-2)}=\frac{2n\alpha}{(n-\alpha)}.
\end{split}
\end{equation*}
Therefore,
\begin{equation*}
\begin{split}
\label{bound A_t,2}
    -(t-s)+A_{t,2}-A_{s,2}&=(t-s)\left(-1+\frac{2n^2}{(n-t)(n-s)}\right)\\
    &=(t-s)\frac{n^2+n(t+s)-ts}{(n-t)(n-s)}\\
    &=\frac{t-s}{2nt}\frac{n^2+n(t+s)-ts}{n-s}A_{t,2}\eqqcolon (1-c_1)\,A_{t,2},
\end{split}
\end{equation*}
where $c_1=c_1(n,s,t)\in(0,1)$. 
This implies, for any $\lambda > c_1$,
\begin{equation*}
    -(t-s)+\lambda A_{t,2}-A_{s,2}>-(t-s)+c_1A_{t,2}-A_{s,2}=0,
\end{equation*}
 proving \eqref{Disuguaglianza step 2}.

\textit{Step 3} We claim that there exists a constant $c_0=c_0(n,s,t)\in(0,1)$ such that, for every $k\geq 2$,
\begin{equation}
\label{eq: main claim step3}
-(t-s)+c_0A_{t,k}-A_{s,k}>0.
\end{equation}
First, we prove that there exists $c_2=c_2(n,s,t)<1$ such that, for any $\lambda\geq c_2$ and for each $k\geq 2$ we have
\begin{equation}\label{eq: claim step3}
    -(t-s)+\lambda A_{t,k}-A_{s,k}>-(t-s)+\lambda A_{t,2}-A_{s,2}.
\end{equation}
This can be deduced by showing that for each $k\geq 2$ we have
\begin{equation*}
    \frac{A_{s,k+1}-A_{s,k}}{A_{t,k+1}-A_{t,k}}<c_2<1.
\end{equation*}
To verify this, recall \eqref{eq: Ak+1-Ak} and observe that
\begin{equation*}\label{Disuguaglianza step 3 for k}
    \frac{A_{s,k+1}-A_{s,k}}{A_{t,k+1}-A_{t,k}}=\frac{s(n+s)}{t(n+t)}\frac{P(s,k)}{P(t,k)}\frac{2k+n-2-t}{2k+n-2-s}<\frac{s(n+s)}{t(n+t)}\eqqcolon c_2 <1,
\end{equation*}
where
\begin{equation*}
    P(\alpha,k)\coloneqq \prod_{j=1}^{k-1}\frac{j+\frac{n+\alpha}{2}}{j+\frac{n-\alpha-2}{2}}.
\end{equation*}
Therefore \eqref{eq: claim step3} is proved.
On the other hand, \eqref{Disuguaglianza step 2} provides the existence of a constant $0<c_1(n,s,t)<1$ such that $-(t-s)+\lambda A_{t,2}-A_{s,2}>0$ for any $\lambda\geq c_1$.
It is therefore sufficient to take $\lambda = c_0(n,s,t)\coloneqq \max\left\{c_1,c_2\right\}$ in \eqref{eq: claim step3} in order to obtain
\begin{equation*}
    -(t-s)+c_0A_{t,k}-A_{s,k}\geq -(t-s)+c_1 A_{t,2}-A_{s,2}>0,
\end{equation*}
namely to prove \eqref{eq: main claim step3}. Notice that $c_0=c_0(n,s,t) \in (0,1)$.
 
\textit{Step 4}. We prove \eqref{eq:coercivity}: we apply to  \eqref{variazione seconda armoniche}, in order, \eqref{eq: main claim step3}, the fact that $A_{t,0}=0$ and $A_{t,1}=t$, and Lemma \ref{lemma control a_0 and a_1 with L_2 norm}, to obtain 
\begin{equation}
\label{coercivity Ht}
\begin{split}
    \delta^2\mathcal{F}(0)[u,u]&\geq\frac{\F(0)}{n\omega_n}\sum_{k\geq 2}(-(t-s)+A_{t,k}-A_{s,k})\sum_{i=0}^{d(k)}a^i_k(u)^2\\
    &\geq \frac{\F(0)}{n\omega_n}\left[(1-c_0)\sum_{k\geq 2}A_{t,k}\sum_{i=0}^{d(k)}a_k^i(u)^2\right]\\
    &=\frac{\mathcal{F}(0)}{n\omega_n}\left[(1-c_0)\sum_{k\geq 0}A_{t,k}\sum_{i=0}^{d(k)}a_k^i(u)^2-(1- c_0)\left(a_0^0(u)^2 A_{t,0}+\sum_{i=0}^{d(1)}a_1^i(u)^2A_{t,1}\right)\right]\\
    &=\frac{\mathcal{F}(0)}{n\omega_n}\left[(1-c_0)\sum_{k\geq 0}A_{t,k}\sum_{i=0}^{d(k)}a_k^i(u)^2-(1- c_0)t\sum_{i=0}^{d(1)}a_1^i(u)^2\right]\\
    &\geq\frac{\mathcal{F}(0)}{n\omega_n}\left[(1-c_0)\|u\|^2_{H^{\frac{1+t}{2}}}-(1-c_0)tC_1(n)\|u\|^4_{L^2}\right]\\
    & \geq\frac{(1-c_0)}{2n\omega_n}\mathcal{F}(0)\|u\|^2_{H^{\frac{1+t}{2}}}\eqqcolon c\|u\|^2_{H^{\frac{1+t}{2}}},
\end{split}
\end{equation}
provided $\|u\|_{L^2}\leq \left(\frac{1}{2C_1(n)}\right)^{\frac{1}{2}}\eqqcolon \eps_0$, where $C_1(n)$ is the constant appearing in Lemma \ref{lemma control a_0 and a_1 with L_2 norm}. \qedhere
\end{proof}

\begin{remark}
    We remark that the factor $c\coloneqq\frac{1-c_0}{2n\omega_n}$ is bounded away from zero whenever $s<t$, also in the regimes $s<t\approx 1$ and $0\approx s<t$: this follows by the fact that
    \begin{equation*}
        c=\frac{1}{2n\omega_n}\min\left\{1-\frac{s(n+s)}{t(n+t)},\frac{t-s}{n-s}\frac{n^2+n(t+s)-ts}{2nt}\right\}
    \end{equation*}
for any $0<s<t<1$. From this, we also deduce that, when $s\approx t$, it holds $(1-c_0)\approx (t-s)$.
\end{remark}
\section{Continuity of the second variation of \texorpdfstring{$\F$}{F}}\label{section continuity}
In this section we estimate the continuity of the second variation of $\F$ along rays and near the origin, that is, for $u\in C^1$ with $\|u\|_{\infty}<\frac{1}{2}$,
\begin{equation}\label{dam-lam estimate}
        \left |\delta^2 \F(u)[u,u]-\delta^2 \F(0)[u,u]\right |\leq C\,\|u\|_{C^1}\|u\|^2_{H^{\frac{1+t}{2}}},
\end{equation}
where $C$ is a constant depending only on $s,t$ and $n$ (Proposition \ref{prop dam-lam estimate}). This estimate allows us to bound the third-order remainder term in the Taylor expansion of $\F$ around $u=0$. 
As an intermediate result, we establish an estimate analogous to \eqref{dam-lam estimate} with $P_\alpha$ in place of $\F$  (Proposition \ref{d^2P_alpha(u)-d^2P_alpha(0)prop}).
As a byproduct, Proposition \ref{d^2P_alpha(u)-d^2P_alpha(0)prop} implies an Alexandrov-type estimate contained in \cite{degenn2023asymptotic}, as we show in Remark \ref{remarkdegennarokubinkubin}.
\begin{proposition}\label{d^2P_alpha(u)-d^2P_alpha(0)prop}
    If $u\in C^1(\dB)$ with $\|u\|_{C^1}<1/2$, and $\alpha\in(0,1)$, there exists $C=C(n,\alpha)>0$ such that    \begin{equation}\label{eq: dam-lam estimate for P_s}
        \left |\delta^2 P_\alpha(u)[u,u]-\delta^2 P_\alpha(0)[u,u]\right |\leq C\,\|u\|_{C^1}\|u\|^2_{H^{\frac{1+\alpha}{2}}}.
    \end{equation}
    Moreover, for any $|\lambda|\leq 1$,
    \begin{equation}\label{eq: dam-lam estimate for P_alpha(lambda u)}
        \left |\delta^2 P_\alpha(\lambda u)[u,u] \right |\leq C\,\|u\|^2_{H^{\frac{1+\alpha}{2}}}.
    \end{equation}
\end{proposition}

\begin{proof}
Throughout this proof, $C$ denotes a positive constant depending only on $n$ and $\alpha$, and may change from line to line.

Observe that \eqref{eq: dam-lam estimate for P_alpha(lambda u)} follows from  \eqref{eq: dam-lam estimate for P_s}. First, for $\lambda=0$ \eqref{eq: dam-lam estimate for P_alpha(lambda u)} follows directly from \eqref{second variation of P_alpha at 0}. For $\lambda\neq 0$, using again \eqref{second variation of P_alpha at 0}, we have
\begin{equation*}
\begin{split}
     \left |\delta^2 P_\alpha(\lambda u)[u,u] \right |&\leq \frac{1}{\lambda^2}\left|\delta^2 P_\alpha(\lambda u)[\lambda u,\lambda u]-\delta^2 P_\alpha(0)[\lambda u,\lambda u] \right |+\left|\delta^2 P_\alpha(0)[ u,u]\right|\\
     &\leq C\frac{1}{\lambda^2}\|\lambda u\|_{C^1}\|\lambda u\|^2_{H^{\frac{1+\alpha}{2}}} + C\|u\|^2_{H^{\frac{1+\alpha}{2}}} \leq \blue{C}\|u\|^2_{H^{\frac{1+\alpha}{2}}}.
\end{split}
\end{equation*}

Now we prove \eqref{eq: dam-lam estimate for P_s}. We divide the proof into four lemmas: in the first, we rewrite the expression we want to estimate:

\begin{lemma}
\label{first lemma}
It holds
\begin{equation}
\label{lemma1eq}
    \delta^2 P_\alpha(u)[u,u]-\delta^2 P_\alpha(0)[u,u]= E_1+E_2+E_3+E_4,
\end{equation}
where
\begin{equation*}
    \begin{split}
        E_1&\coloneqq (n-\alpha)(n-\alpha-1)\frac{P_\alpha(B)}{P(B)}\int_{\partial B} \left((1+u(x))^{n-\alpha-2}-1\right)u(x)^2\,d\sigma,\\
         E_2&\coloneqq 2\int_{\partial B \times \partial B} (u(x)^2-u(x)u(y))\left(F_{|x-y|}(1+u(x),1+u(y))-\frac{1}{|x-y|^{n+\alpha}}\right)\,d\sigma_x\,d\sigma_y,\\
        E_3&\coloneqq 2\int_{\partial B \times \partial B}(u(x)^2 - u(x)u(y))\left(\int_{u(y)}^{u(x)}\partial_1F_{|x-y|}(1+u(x),1+\rho) + \partial_2F_{|x-y|}(1+u(x),1+\rho)\,d\rho\right)\,d\sigma_x\,d\sigma_y,\\
        E_4&\coloneqq \int_{\partial B \times \partial B}u(x)u(y)\left(\int_{u(y)}^{u(x)}G_{|x-y|}(1+u(x),1+\rho)-G_{|x-y|}(1+u(y),1+\rho)\,d\rho\right)\,d\sigma_x\,d\sigma_y,
    \end{split}
\end{equation*}
and
\begin{equation}
\label{G expression}
    G_{|x-y|}(1+u(x),1+\rho)\coloneqq \partial_1F_{|x-y|}(1+u(x),1+\rho) + \partial_2F_{|x-y|}(1+u(x),1+\rho).
\end{equation}
\end{lemma}
\begin{proof}
Recalling \eqref{second variation of P_alpha} we can write
\begin{equation*}
\delta^2 P_\alpha(u)[u,u]-\delta^2 P_\alpha(0)[u,u]=E_1+E_0,
\end{equation*}
where
\begin{equation*}
    \begin{split}
        E_1&=(n-\alpha)(n-\alpha-1)\frac{P_\alpha(B)}{P(B)}\int_{\partial B} \left((1+u(x))^{n-\alpha-2}-1\right)u(x)^2\,d\sigma,\\
        E_0&=2\int_{\partial B \times \partial B} \left[u(x)^2F_{|x-y|}(1+u(x),1+u(x))-u(x)u(y)F_{|x-y|}(1+u(y),1+u(x))\right]\,d\sigma_xd\sigma_y\\
        &-2\int_{\partial B \times \partial B}(u(x)^2-u(x)u(y))\frac{1}{|x-y|^{n+\alpha}}\,d\sigma_xd\sigma_y\\
        &+2\int_{\partial B \times \partial B}u(x)^2\left(\int_{u(y)}^{u(x)}\partial_1F_{|x-y|}(1+u(x),1+\rho)\,d\rho\right)\,d\sigma_xd\sigma_y.
\end{split}    
\end{equation*}
By writing $F_{|x-y|}(1+u(x),1+u(x))=F_{|x-y|}(1+u(x),1+u(y))+\int_{u(x)}^{u(y)}\partial_2F_{|x-y|}(1+u(x),1+\rho)\,d\rho$, and adding and subtracting the term
\begin{equation*}
    2\int_{\partial B \times \partial B} u(x)u(y)\int_{u(y)}^{u(x)}\partial_1 F_{|x-y|}(1+u(x),1+\rho) + \partial_2F_{|x-y|}(1+u(x),1+\rho)\,d\rho,
\end{equation*}
we have
\begin{equation*}
\begin{split}
E_0=&2\int_{\partial B \times \partial B} (u(x)^2-u(x)u(y))\left(F_{|x-y|}(1+u(x),1+u(y))-\frac{1}{|x-y|^{n+\alpha}}\right)\,d\sigma_xd\sigma_y\\
&+2\int_{\partial B \times \partial B}(u(x)^2 - u(x)u(y))\left(\int_{u(y)}^{u(x)}\partial_1F_{|x-y|}(1+u(x),1+\rho) + \partial_2F_{|x-y|}(1+u(x),1+\rho)\,d\rho\right)\,d\sigma_xd\sigma_y\\
&+2\int_{\partial B \times \partial B}u(x)u(y)\left(\int_{u(y)}^{u(x)}\partial_1F_{|x-y|}(1+u(x),1+\rho) + \partial_2F_{|x-y|}(1+u(x),1+\rho)\,d\rho\right)\,d\sigma_xd\sigma_y\\
=&E_2+E_3+E_4,
\end{split}
\end{equation*}
where we used in the last line the symmetry of the integrals with respect to the variables $x$ and $y$.
\end{proof}
In the following lemma, we estimate the first three terms in the right hand side of \eqref{lemma1eq}: 

\begin{lemma}
\label{second lemma}
    For $i=1,2,3$, it holds
    \begin{equation*}
        \left|E_i\right|\leq C\|u\|_{C^1}[u]^2_{H^{\frac{1+\alpha}{2}}}.
    \end{equation*}
\end{lemma}
\begin{proof}
Regarding $E_1$, we observe that 
\begin{equation}\label{term 1}
    |E_1|=(n-\alpha)(n-\alpha-1)\frac{P_\alpha(B)}{P(B)}\int_{\partial B} \left((1+u(x))^{n-\alpha-2}-1\right)u(x)^2\,d\sigma\leq C\|u\|_{\infty}\|u\|^2_{L^2}.
\end{equation}

In order to estimate $E_2$, notice that
\begin{equation*}
    \left|F_{|x-y|}(1+u(x),1+u(y))-\frac{1}{|x-y|^{n+\alpha}}\right|\leq C\frac{\|u\|_{C^1}}{|x-y|^{n+\alpha}}.
\end{equation*}
Indeed, arguing as in \cite[Lemma 2.2]{degenn2023asymptotic}, one can write
\begin{equation*}
    \begin{split}
        F_{|x-y|}(1+u(x),1+u(y))=&\frac{(1+(n-1)u(x)+\mathcal{O}(u^2(x)))(1+(n-1)u(y)+\mathcal{O}(u^2(y)))}{((u(x)-u(y))^2+(1+u(x))(1+u(y))|x-y|^{2})^{\frac{n+\alpha}{2}}}\\
        =&\frac{1+(n-1)(u(x)+u(y))+\mathcal{O}(u^2(x)+u^2(y))}{|x-y|^{n+\alpha}\left(\left(\frac{u(x)-u(y)}{|x-y|}\right)^2+(1+u(x))(1+u(y))\right)^{\frac{n+\alpha}{2}}}\\
        &\hspace{-3cm}=\frac{1+(n-1-\frac{n+\alpha}{2})(u(x)+u(y))}{|x-y|^{n+\alpha}}-\frac{n+\alpha}{2}\frac{(u(x)-u(y))^2}{|x-y|^{n+\alpha+2}}+\mathcal{O}\left(\frac{(u(x)-u(y))^2}{|x-y|^2}+u^2(x)+u^2(y)\right).
    \end{split}
\end{equation*}
From this we deduce that
\begin{equation*}
    \begin{split}
    \left|F_{|x-y|}(1+u(x),1+u(y))-\frac{1}{|x-y|^{n+\alpha}}\right|\leq \frac{C\|u\|_{C^1}}{|x-y|^{n+\alpha}}.
    \end{split}
\end{equation*}
Therefore, symmetrizing the double integral over the sphere and recalling that the function $(x,y)\mapsto F_{|x-y|}(1+u(x),1+u(y))$ is symmetric,
\begin{equation*}\label{term 2}
\begin{split}
    |E_2|&=\left|2\int_{\partial B \times \partial B} (u(x)^2-u(x)u(y))\left(F_{|x-y|}(1+u(x),1+u(y))-\frac{1}{|x-y|^{n+\alpha}}\right)\,d\sigma_xd\sigma_y\right|\\
    &=\int_{\partial B \times \partial B} (u(x)-u(y))^2\left|F_{|x-y|}(1+u(x),1+u(y))-\frac{1}{|x-y|^{n+\alpha}}\right|\,d\sigma_xd\sigma_y\leq C\|u\|_{C^1}[u]^2_{H^{\frac{1+\alpha}{2}}}.
    \end{split}
\end{equation*}

For the term $E_3$, recalling \eqref{F_|x-y|}, we have that
\begin{equation*}
    \partial_1 F_{|x-y|}(a,b)=\frac{a^{n-1}b^{n-1}}{((a-b)^2+ab|x-y|^2)^{\frac{n+\alpha}{2}}}\left(\frac{n-1}{a}-\frac{n+\alpha}{2}\frac{2(a-b)+b|x-y|^2}{((a-b)^2+ab|x-y|^2)}\right),
\end{equation*}
for any $a,b>0$. Notice that, by symmetry, $\partial_2 F_{|x-y|}(a,b)=\partial_1 F_{|x-y|}(b,a)$. Therefore
\begin{equation*}\label{G(a,b)}
    \begin{split}
        G_{|x-y|}(a,b)&= \partial_1 F_{|x-y|}(a,b)+\partial_1F_{|x-y|}(b,a)\\
        &=\frac{a^{n-1}b^{n-1}}{|x-y|^{n+\alpha}\left(\frac{(a-b)^2}{|x-y|^2}+ab\right)^{\frac{n+\alpha}{2}}}\left((n-1)\frac{a+b}{ab}-\frac{n+\alpha}{2\left(\frac{(a-b)^2}{|x-y|^2}+ab\right)}(a+b)\right).
    \end{split}
\end{equation*}
Thus
\begin{equation*}
    \left\|G_{|x-y|}(1+u(x),1+\rho)\right\|_\infty\leq C\frac{1}{|x-y|^{n+\alpha}}.
\end{equation*}
We hence deduce the following bound:
\begin{equation}\label{term 3}
|E_3|\leq\int_{\partial B \times \partial B}|u(x)| |u(x)- u(y)|^2\,\|G_{|x-y|}(1+u(x),1+\rho)\|_{\infty}\,d\sigma_xd\sigma_y\leq C\|u\|_{\infty}[u]^2_{H^{\frac{1+\alpha}{2}}}.
\end{equation}
\end{proof}
The last two lemmas are used to estimate the term $|E_4|$. In Lemma \ref{AuxLemma} we decompose this term in two terms, controlling the first one as desired. The second one, instead, will be the object of Lemma \ref{fourth lemma}.
\begin{lemma}
\label{AuxLemma}
For any $u \in C^1(\partial B)$ with $\|u\|_{C^1}<1$,  it holds
\begin{equation}
\begin{split}
\label{auxlemmaeq}
|E_4|
&\leq C\|u\|^2_{\infty}[u]^2_{H^{\frac{1+\alpha}{2}}}+\int_{\partial B \times \partial B}u(x)u(y)\frac{u(x)-u(y)}{|x-y|^{n+\alpha}}h(u(x),u(y),n,\alpha)\,d\sigma_xd\sigma_y,
\end{split}
\end{equation}
where $h$ is an explicit function.
\end{lemma}
\begin{proof}
Letting $\sigma\coloneqq 1+\frac{u(x)+u(y)}{2}$, $\delta\coloneqq \frac{u(y)-u(x)}{2}$, $t=\frac{2\rho-u(x)-u(y)}{u(y)-u(x)}$, and recalling \eqref{G expression}, we have
\begin{equation}
\begin{split}  
\label{G variabili cambiate}
&\int_{u(y)}^{u(x)}[G(1+u(x),1+\rho) - G(1+u(y),1+\rho)]\,d\rho\\=&\int_{-1}^{1}\left[G(\sigma-\delta,t\delta+\sigma)-G(\sigma+\delta,t\delta+\sigma)\right]\delta\,dt\\
    =&\int_{-1}^{1} \delta \frac{(\sigma-\delta)^{n-1}(t\delta+\sigma)^{n-1}}{D_1^{\frac{n+\alpha}{2}}}\Bigg[ (n-1)\frac{2\sigma+(t-1)\delta}{(\sigma-\delta)(\sigma+t\delta)}-\frac{n+\alpha}{2}\frac{2\sigma+(t-1)\delta}{D_1}|x-y|^2\Bigg]\,dt\\
     -&\int_{-1}^{1}  \delta\frac{(\sigma+\delta)^{n-1}(t\delta+\sigma)^{n-1}}{D_2^{\frac{n+\alpha}{2}}}
     \Bigg[ (n-1)\frac{2\sigma+(t+1)\delta}{(\sigma+\delta)(\sigma+t\delta)}-\frac{n+\alpha}{2}\frac{2\sigma+(t+1)\delta}{D_2}|x-y|^2\Bigg]\,dt,\\
    \end{split}
\end{equation} where, for the sake of clarity, we defined
\begin{equation}
\begin{split}
\label{D_1D_2}
    D_1&\coloneqq (t+1)^2\delta^2+(\sigma-\delta)(t\delta+\sigma)|x-y|^2,\\
    D_2&\coloneqq (t-1)^2\delta^2+(\sigma+\delta)(t\delta+\sigma)|x-y|^2.
    \end{split}
\end{equation}
Now, in \eqref{G variabili cambiate}, we can group together the terms with a factor $\delta^2$ and those with a factor $\delta$, which will contribute to give, respectively, the first and the second term in the right hand side of \eqref{auxlemmaeq}. Namely, we can write
\begin{equation}
\label{maineqauxlemma}
    \int_{u(y)}^{u(x)}[G_{|x-y|}(1+u(x),1+\rho) - G_{|x-y|}(1+u(y),1+\rho)]\,d\rho=G_1+G_2
\end{equation}
where
\begin{equation}
\label{maineqauxlemma2}
\begin{split}
      G_1= &\delta^2\int_{-1}^{1} \frac{(\sigma-\delta)^{n-1}(t\delta+\sigma)^{n-1}}{D_1^{\frac{n+\alpha}{2}}}(t-1)\Big[\frac{n-1}{(\sigma-\delta)(\sigma+t\delta)}
   - \frac{n+\alpha}{2D_1}|x-y|^2\Big]\,dt\\
&-\delta^2\int_{-1}^{1} \frac{ (\sigma+\delta)^{n-1}(t\delta+\sigma)^{n-1}}{D_2^{\frac{n+\alpha}{2}}}(t+1)\Big[\frac{n-1}{(\sigma+\delta)(\sigma+t\delta)}
   - \frac{n+\alpha}{2D_2}|x-y|^2\Big]\,dt,\\
G_2=&\delta\int_{-1}^{1} \frac{(\sigma-\delta)^{n-1}(t\delta+\sigma)^{n-1}}{D_1^{\frac{n+\alpha}{2}}}2\sigma\Big[\frac{n-1}{(\sigma-\delta)(\sigma+t\delta)}
   - \frac{n+\alpha}{2D_1}|x-y|^2\Big]\,dt\\
&-\delta\int_{-1}^{1} \frac{ (\sigma+\delta)^{n-1}(t\delta+\sigma)^{n-1}}{D_2^{\frac{n+\alpha}{2}}}2\sigma\Big[\frac{n-1}{(\sigma+\delta)(\sigma+t\delta)}
   - \frac{n+\alpha}{2D_2}|x-y|^2\Big]\,dt.\\
   \end{split}
\end{equation}
Recalling that $\delta=\frac{1}{2}(u(y)-u(x))$ and 
observing that, by \eqref{D_1D_2}
\begin{equation*}
\begin{split}
        D_1\geq (\sigma-\delta)(t\delta+\sigma)|x-y|^2=(1+u(x))(1+\rho)|x-y|^2\geq (1-\|u\|_{\infty})^2|x-y|^2,
\end{split}
\end{equation*} and the same holds for $D_2$, one can deduce that
\begin{equation*}
    G_1\leq \frac{|u(x)-u(y)|^2}{|x-y|^{n+\alpha}}C.
\end{equation*}
On the other hand,
\begin{equation*}
    G_2=\frac{u(x)-u(y)}{|x-y|^{n+\alpha}}\,h(u(x),u(y),n,\alpha), 
\end{equation*}
where
\begin{equation}
\label{C_1}
\begin{split}
    h(u(x),u(y),n,\alpha)=&\int_{-1}^{1} \frac{(\sigma-\delta)^{n-1}(t\delta+\sigma)^{n-1}}{D_1^{\frac{n+\alpha}{2}}}2\sigma|x-y|^{n+\alpha}\Big[\frac{n-1}{(\sigma-\delta)(\sigma+t\delta)}
   - \frac{n+\alpha}{2D_1}|x-y|^2\Big]\,dt\\
   &-\int_{-1}^{1} \frac{ (\sigma+\delta)^{n-1}(t\delta+\sigma)^{n-1}}{D_2^{\frac{n+\alpha}{2}}}2\sigma|x-y|^{n+\alpha}\Big[\frac{n-1}{(\sigma+\delta)(\sigma+t\delta)}
   - \frac{n+\alpha}{2D_2}|x-y|^2\Big]\,dt.
\end{split}
\end{equation}
\end{proof}
We proceed to estimate the remaining term in \eqref{auxlemmaeq}:
\begin{lemma}
\label{fourth lemma}
    For $h$ defined as in \eqref{C_1}, it holds 
    \begin{equation*}
        h \leq C|u(x)-u(y)|.
    \end{equation*}
    In particular, we deduce
    \begin{equation*}
        |E_4|\leq C\|u\|^2_{\infty}[u]^{2}_{H^{\frac{1+\alpha}{2}}}.
    \end{equation*}
\end{lemma}
\begin{proof}
We need to show that the term $h$ as described in \eqref{C_1} can in turn be estimated as
\begin{equation}
\label{C2andC3}
    h\leq C|\delta|.
\end{equation}
To see this, we first split $h$ in two parts, namely we write
\begin{equation}
\label{C_2andgamma}
    h=2\sigma(n-1)|x-y|^{n+\alpha}\gamma_1+\sigma(n+\alpha)|x-y|^{n+\alpha}\gamma_2,
\end{equation}
where
\begin{equation*}
\begin{split}
    \gamma_1&=\int_{-1}^{1} \frac{(\sigma-\delta)^{n-2}(t\delta+\sigma)^{n-2}}{D_1^{\frac{n+\alpha}{2}}}\,dt
   -\int_{-1}^{1} \frac{ (\sigma+\delta)^{n-2}(t\delta+\sigma)^{n-2}}{D_2^{\frac{n+\alpha}{2}}}\,dt,\\
   \gamma_2&=-\int_{-1}^{1} \frac{(\sigma-\delta)^{n-1}(t\delta+\sigma)^{n-1}}{D_1^{\frac{n+\alpha}{2}+1}}|x-y|^2\,dt
    +\int_{-1}^{1} \frac{ (\sigma+\delta)^{n-1}(t\delta+\sigma)^{n-1}}{D_2^{\frac{n+\alpha}{2}+1}}|x-y|^2\,dt.
   \end{split}
\end{equation*}
We deal first with the term $\gamma_1$: recalling the definition \eqref{D_1D_2}, we can add and subtract the terms
\begin{equation*}
    \begin{split}
         A_1\coloneqq&\int_{-1}^{1} \frac{(\sigma+\delta)^{n-2}(t\delta+\sigma)^{n-2}}{((t+1)^2\delta^2+(\sigma+\delta)(t\delta+\sigma)|x-y|^2)^{\frac{n+\alpha}{2}}}\,dt,\\
         A_2\coloneqq &\int_{-1}^{1} \frac{(\sigma+\delta)^{n-2}(-t\delta+\sigma)^{n-2}}{((-t+1)^2\delta^2+(\sigma+\delta)(-t\delta+\sigma)|x-y|^2)^{\frac{n+\alpha}{2}}}\,dt,
    \end{split}
\end{equation*}
and write
\begin{equation*}
\label{gamma_1}
\begin{split}
    \gamma_1=&\int_{-1}^{1} \frac{(\sigma-\delta)^{n-2}(t\delta+\sigma)^{n-2}}{D_1^{\frac{n+\alpha}{2}}}\,dt
   -A_1\\
   &+A_1-A_2\\
   &+A_2-\int_{-1}^{1} \frac{(\sigma+\delta)^{n-2}(t\delta+\sigma)^{n-2}}{D_2^{\frac{n+\alpha}{2}}}\,dt.
   \end{split}
\end{equation*}
We first notice that, by symmetry, $A_1-A_2=0$.

Then, we treat the term 
\begin{equation}
\begin{split}
\label{differenza gamma1}
    B_1\coloneqq &\int_{-1}^1\frac{(\sigma-\delta)^{n-2}(t\delta+\sigma)^{n-2}}{D_1^{\frac{n+\alpha}{2}}}\,dt-A_1\\
    &=\int_{-1}^1 \Bigg[\frac{(\sigma-\delta)^{n-2}(t\delta+\sigma)^{n-2}}{((t+1)^2\delta^2+(\sigma-\delta)(t\delta+\sigma)|x-y|^2)^{\frac{n+\alpha}{2}}}\\
    &- \frac{(\sigma+\delta)^{n-2}(t\delta+\sigma)^{n-2}}{((t+1)^2\delta^2+(\sigma+\delta)(t\delta+\sigma)|x-y|^2)^{\frac{n+\alpha}{2}}}\Bigg]\,dt.
\end{split}
\end{equation}
Introduce the function 
\begin{equation}
\label{auxiliaryfunctiong}
g(x)\coloneqq \frac{x^{n-2}k}{(c+xm)^{\frac{n+\alpha}{2}}},
\end{equation}
where $k=(t\delta+\sigma)^{n-2}$, $c= (t+1)^2\delta^2$, $m= (t\delta+\sigma)|x-y|^2$, so that $B_1=\int_{-1}^{1}\left[g(\sigma-\delta)-g(\sigma+\delta)\right]$.
Using $|g(x_1)-g(x_2)|\leq \sup_{(x_1,x_2)}|g'||x_1-x_2|$, we can estimate
\begin{equation*}
\begin{split}
|g(\sigma+\delta)-g(\sigma-\delta)|\leq& \sup_{x\in(\sigma-\delta,\sigma+\delta)}\bigg|\frac{d}{dx}\frac{x^{n-2}k}{(c+xm)^{\frac{n+\alpha}{2}}}\bigg|2|\delta|\\
=& 2|\delta|\sup_{x\in(\sigma-\delta,\sigma+\delta)}\bigg|\frac{(n-2)x^{n-3}k}{(c+xm)^{\frac{n+\alpha}{2}}}-\frac{n+\alpha+2}{2}\frac{x^{n-2}km}{(c+xm)^{\frac{n+\alpha}{2}+1}}\bigg|\\
\leq & \frac{C}{|x-y|^{n+\alpha}}|\delta|.
\end{split}
\end{equation*}
Therefore
\begin{equation*}
    \begin{split}
        |B_1|= \left|\int_{-1}^{1}g(\sigma+\delta)-g(\sigma-\delta)\,dt\right|\leq C\frac{|\delta|}{|x-y|^{n+\alpha}}.
    \end{split}
\end{equation*}
Similarly, in order to estimate 
\begin{equation*}
\begin{split}
B_3&\coloneqq A_2-\int_{-1}^{1} \frac{(\sigma+\delta)^{n-2}(t\delta+\sigma)^{n-2}}{D_2^{\frac{n+\alpha}{2}}}\,dt\\
&=\int_{-1}^{1} \left[\frac{(\sigma+\delta)^{n-2}(-t\delta+\sigma)^{n-2}}{((-t+1)^2\delta^2+(\sigma+\delta)(-t\delta+\sigma)|x-y|^2)^{\frac{n+\alpha}{2}}}\,dt- \frac{(\sigma+\delta)^{n-2}(t\delta+\sigma)^{n-2}}{((t-1)^2\delta^2+(\sigma+\delta)(t\delta+\sigma)|x-y|^2)^{\frac{n+\alpha}{2}}}\right]\,dt,
\end{split}
\end{equation*}
we can use the same auxiliary function introduced in  \eqref{auxiliaryfunctiong}, with $k=(\sigma+\delta)^{n-2}$, $c=(-t+1)^2\delta^2$ and $m=(\sigma+\delta)|x-y|^2$. In this way, we get
\begin{equation*}
    \left|B_3\right|\leq C\frac{|\delta|}{|x-y|^{n+\alpha}},
\end{equation*}
and hence
\begin{equation}
\label{gamma_1ineq}
    |\gamma_1|\leq |\delta|C.
\end{equation}

To conclude, we notice that one can adapt the estimates used to obtain \eqref{gamma_1ineq} for $\gamma_2$ by repeating verbatim the argument,
which in turn gives
\begin{equation}
\label{gamma_2ineq}
    |\gamma_2|\leq |\delta|C.
\end{equation}
Recalling \eqref{C_2andgamma}, by estimates \eqref{gamma_1ineq} and \eqref{gamma_2ineq} we deduce \eqref{C2andC3}, which provides the desired estimate on the term $E_4$:
\begin{equation*}
\begin{split}
    |E_4|
&\leq C\|u\|^2_{\infty}[u]^2_{H^{\frac{1+\alpha}{2}}}+\int_{\partial B \times \partial B}u(x)u(y)\frac{u(x)-u(y)}{|x-y|^{n+\alpha}}h(u(x),u(y),n,\alpha)\,d\sigma_xd\sigma_y\\
    &\leq  C\|u\|^2_{\infty}[u]^2_{H^{\frac{1+\alpha}{2}}}+C\int_{\partial B \times \partial B}u(x)u(y)\frac{|u(x)-u(y)|^2}{|x-y|^{n+\alpha}}\,d\sigma_xd\sigma_y\\
    &\leq C\,\|u\|^2_{\infty}[u]^2_{H^{\frac{1+\alpha}{2}}}.
\end{split}
\end{equation*}
This concludes the proof of Lemma \ref{fourth lemma} and hence of Proposition \ref{d^2P_alpha(u)-d^2P_alpha(0)prop}.
\renewcommand{\qedsymbol}{}
\end{proof}
\end{proof}

\begin{remark}
    Notice that the constant $C(n,\alpha)$ in Proposition \ref{d^2P_alpha(u)-d^2P_alpha(0)prop} is finite as $\alpha\to0$ or $\alpha\to1$.
\end{remark}

\begin{remark}
\label{remarkdegennarokubinkubin}
    Before continuing our discussion, let us mention that Proposition \ref{d^2P_alpha(u)-d^2P_alpha(0)prop} implies Theorem 0.3 obtained in \cite{degenn2023asymptotic}: indeed, on the one hand, if $\|u\|_{C^1}$ is small enough (and hence $\|u\|_{L^2}$), by Proposition \ref{prop coercivity} and Proposition \ref{d^2P_alpha(u)-d^2P_alpha(0)prop}, we deduce
    \begin{equation*}
    \begin{split}
        \delta P_\alpha(u)[u]=&\,\delta P_\alpha (0)[u]+\int_0^1\delta^2P_\alpha(u)[u,u]\, dt\\
        =& \int_0^1 \delta^{2}P_\alpha(0)[u,u]+\delta^2P_\alpha(0)[u,u]-\delta^2P_\alpha(u)[u,u]\,dt\\
        \geq&\,\int_0^1 \delta^2 P_\alpha(0)[u,u]-\left|\delta^2P_\alpha(0)[u,u]-\delta^2P_\alpha(u)[u,u]\right|\,dt\\
        \geq&\,c\|u\|^2_{H^{\frac{1+\alpha}{2}}}.
    \end{split}
    \end{equation*}
    On the other hand, denoting by $H^\alpha[u](x)$ the $\alpha$-mean curvature of the set $E_u$ at the point $x \in \partial E_u$ (see \cite{degenn2023asymptotic} for the definition), and by $\overline{H^\alpha}[u]$ the average of $H^\alpha[u]$ over $\partial E_u$, we have
    \begin{equation*}
        \delta P^\alpha (u)[u]=\int_{\partial B} \left(H^\alpha [u](x)-\overline{H^\alpha}[u]\right)u(x)\,dx\leq \frac{1}{\lambda}\left\|H^\alpha[u]-\overline{H^\alpha}[u]\right\|^2_{L^2}+\lambda \|u\|^2_{L^2},
    \end{equation*}
    for any $\lambda>0$, and hence we deduce the following Alexandrov - type estimate (Theorem 0.3 in \cite{degenn2023asymptotic})
    \begin{equation*}
        \|u\|^2_{H^{\frac{1+\alpha}{2}}}\leq C(n,\alpha)\left\|H^{\alpha}[u]-\overline{H^{\alpha}}[u]\right\|^2_{L^2}.
    \end{equation*}
\end{remark}

We conclude the section adapting the continuity property of Proposition \ref{d^2P_alpha(u)-d^2P_alpha(0)prop} to the functional $\mathcal{F}$:
\begin{proposition}\label{prop dam-lam estimate}
    If $u\in C^1(\dB)$ with $\|u\|_{C^1}\leq\frac{1}{2}$ and $|E_u|=|B|$, then there exists $C=C(n,s,t)>0$ such that
    \begin{equation}\label{eq: dam-lam estimate}
        \left |\delta^2 \F(u)[u,u]-\delta^2 \F(0)[u,u]\right |\leq C\|u\|_{C^1}\|u\|^2_{H^{\frac{1+t}{2}}}.
        \end{equation}
\end{proposition}

\begin{proof} 
Fix $u$ as in the hypothesis and let us define the scalar functions 
$p_{t,u},p_{s,u}, f_{u}\colon [0,1]\to [0,\infty)$ by
\begin{equation*}
    p_{t,u}(\lambda)=P_t(\lambda u),\quad\quad p_{s,u}(\lambda)=P_s(\lambda u)\quad\text{and}\quad f_u(\lambda)=\frac{p_{t,u}(\lambda )^{\frac{1}{n-t}}}{p_{s,u}(\lambda )^{\frac{1}{n-s}}}.
\end{equation*}
Recalling \eqref{second variation of F} we have
\begin{equation*}
\begin{split}
f_u''(\lambda) &= f_u(\lambda)\left[ 2 \left(\frac{p_{t,u}'(\lambda)}{(n-t) p_{t,u}(\lambda)}\right)^2 
- 2 \frac{p_{t,u}'(\lambda)}{(n-t) p_{t,u}(\lambda)} \frac{p_{s,u}'(\lambda)}{(n-s) p_{s,u}(\lambda)} \right] \\
&\quad + f_u(\lambda) \left[ \frac{p_{t,u}''(\lambda)}{(n-t) p_{t,u}(\lambda)} - \frac{p_{s,u}''(\lambda)}{(n-s) p_{s,u}(\lambda)} \right].
\end{split}
\end{equation*}
Therefore, by triangle inequality
\begin{equation}\label{eq0 prop stima d^2F(u)-d^2F(0)}
    \begin{split}
      &\left|\delta^2 \F(u)[u,u]-\delta^2 \F(0)[u,u]\right| = \left| f_u''(1) - f_u''(0) \right| \\
      &\quad\quad\leq
\frac{2 |f_u(1)|}{(n-t)^2 |p_{t,u}(1)|^2} \left| p_{t,u}'(1) \right|^2 + \frac{2 |f_u(1)|}{(n-s)(n-t) |p_{s,u}(1)| |p_{t,u}(1)|} \left| p_{s,u}'(1) \right| \left| p_{t,u}'(1) \right| \\
&\quad\quad
+ \frac{2 |f_u(0)|}{(n-t)^2 |p_{t,u}(0)|^2} \left| p_{t,u}'(0) \right|^2 + \frac{2 |f_u(0)|}{(n-s)(n-t) |p_{s,u}(0)| |p_{t,u}(0)|} \left| p_{s,u}'(0) \right| \left| p_{t,u}'(0) \right| \\
&\quad\quad+ \left| f_u(1) \frac{p_{t,u}''(1)}{(n-t) p_{t,u}(1)} - f_u(0) \frac{p_{t,u}''(0)}{(n-t) p_{t,u}(0)} \right|
+ \left| f_u(1) \frac{p_{s,u}''(1)}{(n-s) p_{s,u}(1)} - f_u(0) \frac{p_{s,u}''(0)}{(n-s) p_{s,u}(0)} \right|.
    \end{split}
    \end{equation}
First, we estimate the terms where only the first derivative of the perimeter appears. Using \eqref{first var of P_alpha at zero}, the volume constraint, and Lemma \ref{lemma control a_0 and a_1 with L_2 norm}, we have
\begin{equation}\label{eq1 prop stima d^2F(u)-d^2F(0)}
    |p_{\alpha,u}'(0)|\leq C(n,\alpha)\|u\|^2_{L^2} \quad\text{for }\alpha=s,t.
\end{equation}
Throughout this proof $C(n,\alpha)$ denotes a positive constant that may vary from line to line.
    By Taylor expansion, there exists $\tilde\lambda \in(0,1)$ such that $   p_{\alpha,u}'(1)=p_{\alpha,u}'(0)+p_{t,u}''(\tilde{\lambda}).$
    Hence, by \eqref{eq: dam-lam estimate for P_alpha(lambda u)} and \eqref{eq1 prop stima d^2F(u)-d^2F(0)}, we obtain
\begin{equation}\label{eq2 prop stima d^2F(u)-d^2F(0)}
         |p_{\alpha,u}'(1)|\leq C(n,\alpha)\|u\|^2_{L^2}.
    \end{equation}
Using \eqref{eq1 prop stima d^2F(u)-d^2F(0)} and \eqref{eq2 prop stima d^2F(u)-d^2F(0)} we get that the first four terms in \eqref{eq0 prop stima d^2F(u)-d^2F(0)} are bounded by
$C(n,s,t)\|u\|^4_{L^2}$.
    
Next, we estimate the last two terms in \eqref{eq0 prop stima d^2F(u)-d^2F(0)}. By triangle inequality, we obtain
\begin{equation}\label{eq5 prop stima d^2F(u)-d^2F(0)}
\begin{split}
&\left| f_u(1) \frac{}{} \frac{p_{\alpha,u}''(1)}{(n-\alpha) p_{\alpha,u}(1)} - f_u(0) \frac{p_{\alpha,u}''(0)}{(n-\alpha) p_{\alpha,u}(0)} \right|
\leq \frac{|f_u(1)|}{(n-\alpha)|p_{\alpha,u}(1)|} \left| p_{\alpha,u}''(1) - p_{\alpha,u}''(0) \right| \\
&\quad + \frac{|f_u(1)|}{n-\alpha} \left| p_{\alpha,u}''(0) \right| \left| \frac{1}{p_{\alpha,u}(1)} - \frac{1}{p_{\alpha,u}(0)} \right|
+ \frac{|p_{\alpha,u}''(0)|}{(n-\alpha) |p_{\alpha,u}(0)|} \left| f_u(1) - f_u(0) \right|.
\end{split}
\end{equation}
By Taylor's formula, as before, by \eqref{eq: dam-lam estimate for P_alpha(lambda u)} and \eqref{eq1 prop stima d^2F(u)-d^2F(0)} we have
\begin{equation}\label{eq6 prop stima d^2F(u)-d^2F(0)}
|p_{\alpha,u}(1) - p_{\alpha,u}(0)| = \left| p_{\alpha,u}'(0) + p_{\alpha,u}''(\tilde\lambda) \right| \leq C(n,\alpha) \|u\|^2_{H^{\frac{1+\alpha}{2}}}, \quad \tilde\lambda\in (0,1).
\end{equation}
Moreover, since the function $(x,y) \mapsto x^{\frac{1}{n-t}} y^{-\frac{1}{n-s}}$ is locally Lipschitz away from $y=0$, we obtain
\begin{equation}\label{eq4 prop stima d^2F(u)-d^2F(0)}
|f_u(1) - f_u(0)| \leq C \left( |p_{t,u}(1) - p_{t,u}(0)| + |p_{s,u}(1) - p_{s,u}(0)| \right)
\leq C(n,s,t) \|u\|^2_{H^{\frac{1+\alpha}{2}}}.
\end{equation}
Combining \eqref{eq: dam-lam estimate for P_s}, \eqref{second variation of P_alpha at 0}, \eqref{eq6 prop stima d^2F(u)-d^2F(0)}, and \eqref{eq4 prop stima d^2F(u)-d^2F(0)} we can bound \eqref{eq5 prop stima d^2F(u)-d^2F(0)} as
\begin{equation}\label{eq7 prop stima d^2F(u)-d^2F(0)}
\begin{split}
    \left| f_u(1) \frac{}{} \frac{p_{\alpha,u}''(1)}{(n-\alpha) p_{\alpha,u}(1)} - f_u(0) \frac{p_{\alpha,u}''(0)}{(n-\alpha) p_{\alpha,u}(0)} \right|
    &\leq C(n,\alpha)\left(\|u\|_{C^1}\|u\|^2_{H^{\frac{1+t}{2}}}+\|u\|^4_{L^2}\right)\\ 
    &\leq C(n,\alpha)\|u\|_{C^1}\|u\|^2_{H^{\frac{1+t}{2}}}.
\end{split}
\end{equation}
Finally, combining \eqref{eq1 prop stima d^2F(u)-d^2F(0)},\eqref{eq2 prop stima d^2F(u)-d^2F(0)} and \eqref{eq7 prop stima d^2F(u)-d^2F(0)}, we conclude that
\begin{equation*}
    \left|\delta^2 \F(u)[u,u]-\delta^2 \F(0)[u,u]\right |\leq C(n,s,t)\|u\|_{C^1}\|u\|^2_{H^{\frac{1+t}{2}}}.
\end{equation*}
\end{proof}

\section{Proof of Theorems \ref{main theorem} and \ref{main theorem quantitative}}\label{section geometric results}

We first prove Theorem \ref{main theorem quantitative} using coercivity  and continuity  of $\delta^2\F(0)$ (Propositions \ref{prop coercivity} and \ref{prop dam-lam estimate}, respectively).
In order to prove Theorem \ref{main theorem}, we need to remove the constraints on the barycenter and the volume of $E_u$ and this is achieved using Lemma \ref{lemma graph with different volume and baricenter}.

\begin{proof}[Proof of Theorem \ref{main theorem quantitative}]
Since the functional $\F$ is invariant under rescaling and translations, we may assume without loss of generality that $E$ is the graph over the unit ball $B$.

Consider the real function $\lambda\mapsto f(\lambda)=F(\lambda u)$ for $\lambda\in[0,1]$. By \eqref{eq: dam-lam estimate}, $f''(\lambda)$ exists and it is bounded for all $\lambda\in[0,1]$. The Taylor expansion of $f$ near $0$ gives
    \begin{equation*}
        f(1)-f(0)=f'(0)+\int_0^1(1-t)f''(t)dt.
    \end{equation*}
    Equivalently, in terms of the functional $\F$, using also \eqref{first variation =0}, we have 
    \begin{equation*}
        \mathcal{F}(u)-\mathcal{F}(0)=\delta \mathcal{F}(0)[u]+\int_0^1(1-t)\delta^2\mathcal{F}(tu)[u,u]dt=\int_0^1(1-t)\delta^2\mathcal{F}(tu)[u,u]dt.
    \end{equation*}
    Then, by \eqref{eq:coercivity} and \eqref{eq: dam-lam estimate}, we estimate
    \begin{equation*}
    \begin{split}
        \delta^2\mathcal{F}(tu)[u,u]&\geq  \delta^2\mathcal{F}(0)[u,u]-\left|\delta^2\mathcal{F}(tu)[u,u]-\delta^2\mathcal{F}(0)[u,u]\right|\\
        &\geq (c-C\|u\|_{C^1})\|u\|^2_{H^{\frac{1+t}{2}}}.
        \end{split}
    \end{equation*}
    For $\|u\|_{C^1}$ sufficiently small, this concludes the proof.
\end{proof}

\begin{lemma}\label{lemma graph with different volume and baricenter}
     Let $E$ be a graph over any ball via the function $u$, with $\|u\|_{C^1(\partial B)}\leq\eps$. Then there exists a constant $c=c(n)>0$ such that $E$ is also a graph over $B(y,r)$, the unique ball with $\bari(E)=y$ and $|E|=|B(y,r)|$, via a function $v$ satisfying $\|v\|_{C^1(\partial B)}\leq c(n)\eps$.
\end{lemma}

\begin{proof}
Without loss of generality, we can assume that $E$ is the graph over the unit ball via the function $u$, thus, $\partial E=\{(1+ u(x))x\colon x\in\partial B\}$. Let $y\in\R^n$ and $r>0$ be such that $\bari(E)=y$ and $|E|=|B(y,r)|$. By \eqref{eq: int(1+u)^n-1=0} and \eqref{eq: xint(1+u)^n+1=0} there exists a constant $c(n)>0$ such that 
\begin{equation*}
    |y|\leq c(n)\|u\|_{C^1}\quad\text{and}\quad |r-1|\leq c(n)\|u\|_{C^1}.
\end{equation*}
We aim to show that there exists a scalar function $v$ with $\|v\|_{C^1(\partial B)}\leq c(n)\eps$ and such that
\begin{equation*}
    \partial E=\{r(1+v(\xi))\xi+y\colon \xi\in\partial B \}.
\end{equation*}
To this end, defining the map $z\colon \partial B\to \partial B$ as
    \begin{equation*}
    z(x) \coloneqq\frac{(1+u(x))x-y}{|(1+u(x))x-y|},
    \end{equation*}
    and    \begin{equation}\label{equation of v(z(w))}
    v(z(x))\coloneqq \frac{1}{r}|(1+u(x))x-y|-1,
    \end{equation}
we obtain that
\begin{equation*}
    (1+ u(x))x=r(1+v(z(x)))z(x)+y \quad\text{for any }x\in\partial B.
\end{equation*}

Assume that $z$ is a local diffeomorphism. Since $\partial B$ is compact and connected, $z$ is a surjective covering map. Moreover, since $z$ is homotopic to the identity, it has degree one and is therefore injective. It follows that $z$ is a global diffeomorphism. 
Thus, by \eqref{equation of v(z(w))}, the function $v\colon \partial B\to\R$ is well-defined and of class $C^1$, and satisfies 
\begin{equation*}
    \|v\|_{C^1}\leq \frac{1}{1-c(n)\|u\|_{C^1}}(1+2\|u\|_{C^1})-1\leq c(n)\|u\|_{C^1}.
\end{equation*} 

We now verify that $z$ is a local diffeomorphism. Let $x\in\partial B$ and $\eta \in T_x(\partial B)$ be a unit tangent vector to the sphere at $x$. If we denote by $h(x)\coloneqq(1+u(x))x-y$, we have

\begin{equation*}
    \partial_\eta h(x)=x\partial_\eta u(x)+ (1+u(x))\eta,
\end{equation*}
and, since $x\cdot\eta=0$,
\begin{equation*}
    h(x)\cdot \partial_\eta h(x)=(1+u(x))\partial_\eta u(x)+ y\cdot \left(x\partial_\eta u(x)+ (1+u(x))\eta\right).
\end{equation*}
Therefore
\begin{equation*}
\begin{split}
        \left|\frac{\partial}{\partial \eta} z(x)\right|=\left|\frac{\partial_\eta h(x)}{|h(x)|}-\frac{h(x)}{|h(x)|^3}h(x)\cdot \partial_\eta h(x)\right|\geq \frac{|\partial_\eta h(x)|}{|h(x)|}-\frac{1}{|h(x)|^2}\left|h(x)\cdot \partial_\eta h(x)\right|\geq 1-c(n)\|u\|_{C^1}.
\end{split}
\end{equation*} 
This shows that $z$ has non-vanishing tangential derivatives, and hence it is a local diffeomorphism.
\end{proof}

\begin{proof}[Proof of Theorem \ref{main theorem}]
Let $E$ be the graph over a ball via the function $u$ satisfying $\|u\|_{C^1}\leq\eps_1$. If $B(y,r)$ is the unique ball such that $\bari(E)=y$ and $|E|=|B(y,r)|$, by Lemma \ref{lemma graph with different volume and baricenter} $E$ is also a graph over $B(y,r)$ via the function $v$ satisfying  $\|v\|_{C^1(\partial B)}\leq c(n)\eps_1\leq\eps_0$. Hence, by Theorem \ref{main theorem quantitative}, we have
\begin{equation*}
    \mathcal{F}(E)\geq \mathcal{F}(B)+c_0\|v\|^2_{H^{\frac{1+t}{2}}},
\end{equation*}
and in particular
\begin{equation*}
    \F(E)\geq\F(B).
\end{equation*}
\end{proof}
\begin{comment}
\end{comment}
\printbibliography
\end{document}